\documentclass[11pt]{article}

\usepackage{amssymb}
\usepackage{amsmath}
\usepackage{amsthm}
\usepackage{amsopn}
\usepackage{mathrsfs}
\usepackage{a4}
\usepackage[english]{babel}
\usepackage{harvard}
\usepackage{comment}
\usepackage{graphics}
\usepackage{epsfig}

\input{mymath.sty}

\renewcommand{\cal}{\mathscr}     % More beautiful calligraphic

\providecommand{\MR}{$\spadesuit$}

\begin{document}

\title{Pointwise adaptive estimation\\ for robust and quantile regression\footnote{This research has been partially supported by
a grant {\it Bonus Qualit\'e Recherche 2007} from Universit\'e Paris
Descartes.}}
\author{Markus Rei{\ss}\footnote{Institut f\"ur Mathematik, Humboldt-Universit\"at, Unter den Linden 6, D-10099 Berlin, Germany, {\tt
mreiss@mathematik.hu-berlin.de} (corresponding author)}\and Yves Rozenholc\footnote{Universit\'e Paris
Ren\'e Descartes, {\tt yves.rozenholc@univ-paris5.fr}} \and  Charles A.
Cuenod\footnote{George Pompidou Hospital Paris, {\tt
charles-andre.cuenod@egp.aphp.fr }}}
\date{This Version: \today}
\maketitle

\begin{abstract}
 A nonparametric procedure for robust regression estimation and for quantile regression is proposed which is completely data-driven and adapts locally to the regularity of the regression function. This is achieved by considering in each point M-estimators over different local neighbourhoods and by a local model selection procedure based on sequential testing. Non-asymptotic risk bounds are obtained, which yield rate-optimality for large sample asymptotics under weak conditions. Simulations for different univariate median regression models show good finite sample properties, also in comparison to traditional methods. The approach is extended to image denoising and applied to CT scans in cancer research.
\end{abstract}

\vspace*{3cm}

\footnoterule \noindent {\sl 2000 Mathematics Subject
Classification}. Primary 62G08;
secondary 62G20, 62G35, 62F05, 62P10.

\pagebreak

\section{Introduction}

We consider a generalized regression model
\[ Y_i=g(x_i)+\eps_i,\quad i=1,\ldots,n,\]
with $(\eps_i)$ i.i.d., $x_1,\ldots,x_n$ in the design space ${\cal X}$ and
$g:{\cal X}\to\R$. The problems we have in view are those of robust nonparametric estimation of $g$ in the presence of heavy-tailed noise $(\eps_i)$ and of nonparametric quantile estimation, which is becoming more and more popular in applications. One main application will be robust image denoising. In the spirit of classical M-estimation \cite{Huber} we therefore consider $g(x_i)$ as the location parameter in the observation $Y_i$, that is
\begin{equation}\label{eqfdef}
 g(x_i)=\argmin_{m\in\R}\E[\rho(Y_i-m)]
\end{equation}
for some convex function $\rho:\R\to\R^+$ with $\rho(0)=0$. We shall assume that $g(x_i)$ is uniquely defined by \eqref{eqfdef}, which is true in all cases of interest. If the $Y_i$ have Lebesgue densities, then often an equivalent description is given by the first order condition $\E[\rho'(\eps_i)]=0$ where $\rho'$ denotes the (weak) derivative. Standard examples are $\rho(x)=x^2/2$ for the classical mean regression model ($\E[\eps_i]=0$), $\rho(x)=|x|$ for the median regression model ($\PP(\eps_i\le 0)=\PP(\eps_i\ge 0)=1/2$) and the intermediate case $\rho(x)=x^2/2$ for $\abs{x}\le k$ and $\rho(x)=k\abs{x}-k^2/2$ for $\abs{x}\ge k$ with some $k>0$ for the Huber estimator ($\E[\min(\max(\eps_i,-k),k)]=0$). The quantile regression model is obtained for $\rho(x)=\abs{x}+(2\alpha-1) x$ ($\PP(\eps_i\le 0)=\alpha$ with quantile $\alpha\in (0,1)$), see e.g. \cit{Koenker}. Since we shall care about robustness, we merely assume a mild moment condition $\eps_i\in L^r$ for some $r\ge 1$ and measure the error in $L^r$-norm.

The function $g$ is not supposed to satisfy a global smoothness criterion, but we aim at estimating it locally in each point $x\in{\cal X}$ as efficiently as possible. The risk will then depend on local regularity properties, which we do not assume to be known. For spatially inhomogeneous functions, in the presence of jumps or for image denoising pointwise adaptive methods are much more appropriate than global smoothing methods. In classical mean regression local adaptivity can be achieved using wavelet thresholding or kernels with locally varying bandwidths, see \cit{Lepskietal} for a discussion. In this ideal situation a data-driven choice among linear empirical quantities is performed. M-estimators are typically nonlinear and the standard approaches do not necessarily transfer directly.  \cit{Brownetal}, for example, use an intermediate data binning and then apply wavelet thresholding to the binned data for median regression. On the other hand, \cit{HallJones}, \cit{Portnoy} and \cit{vandeGeer}
consider kernels, smoothing splines and more general $M$-estimation for quantile regression, but they all use global methods for choosing the tuning parameters like cross-validation or penalisation. Here, we develop a generic algorithm to select optimally among local M-estimators. In contrast to classical model selection, we do not only rely on the estimator values themselves to define a data-driven selection criterion. This has significant advantages in the present case of nonlinear base estimators.

Subsequently, we assume that the statistician has chosen the suitable definition of $\rho$ for the problem at hand and we use the corresponding sample versions to construct base estimators for the (generalized) regression function $g$. In the spirit of classical nonparametrics, we assume that $g$ is locally almost constant around a given point $x\in{\cal X}$. The statistical challenge is to select adaptively the right neighbourhood $U$ of $x$ where a local $M$-estimator is applied.
Let us write
\begin{equation}\label{eqmdef}
 m(Y_i,\,x_i\in U):=\arginf_{\mu\in\R}\Big\{\sum_{i:x_i\in U}\rho(Y_i-\mu)\Big\}
\end{equation}
for the location estimator on the set $U\subset{\cal X}$. If the minimizer is not unique, we just select one of them (e.g., a version of the sample median for $\abs{U}$ even). Note that an extension to general local polynomial or more general local-likelihood estimation is straightforward, but this is not the focus of the present work. For each point $x$ let a family
of nested neighbourhoods $U_0\subset U_1\subset\cdots\subset U_K$ be given and set
\[\tilde\theta_k:=m(Y_i,\,x_i\in U_k).
\]
Then the family $(\tilde\theta_k)_{0\le k\le K}$ forms the class of base estimators and we aim at selecting the best estimator of $\theta:=g(x)$ in this family.

\begin{example} \label{ex1}
Let the design space be ${\cal X}=[0,1]$ with equidistant design
points $x_i=i/n$ and take $\rho(x)=\abs{x}$. Consider the symmetric windows
$U_k=[x-h_k,x+h_k]$ generated by some bandwidths $0\le
h_0<h_1<\cdots<h_K$. Then $\tilde\theta_k$ is the classical
median filter, see e.g. \cit{Truong} or \cit{AriasDonoho}.
\end{example}

Using Lepski's approach as a starting point, we present our procedure to select
optimally among local M-estimators in Section \ref{SecProc}. We argue in a
multiple testing interpretation that our procedure is usually more powerful.
Moreover, it is equally simple to analyze and easy to implement. In Sections
\ref{SecError} and \ref{SecAsymp} we derive exact and asymptotic error bounds
and the latter give optimal minimax rates for H\"older classes. The simulations
in Section \ref{SecSimul} show that our procedure has convincing finite sample
properties. Moreover, they confirm that Lepski's classical method applied to
local median estimators suffers from oversmoothing because changes in the
signal are not detected early enough due to the robustness of the median.
Finally, the procedure has been implemented to denoise dynamical CT image
sequences in Section \ref{SecImage}, which is of key interest when assessing
tumor therapies. Two more technical proofs are postponed to Section
\ref{SecApp}.

\section{The procedure}\label{SecProc}

\subsection{Main ideas}

As a starting point let us consider the standard \cit{Lepski} method for
selecting among $(\tilde\theta_k)_{0\le k\le K}$, given the mean regression
model with $\E[\eps_i]=0$ and $\E[\eps_i^2]<\infty$. Note that the base
estimators are then ordered according to decreasing variances:
$\Var(\tilde\theta_k)\le \Var(\tilde\theta_{k-1})$. On the other hand, the bias
is usually increasing with increasing neighbourhoods $U_k$. This is not always
the case (for example, think of local means for a linear $g$), but true in
particular for the worst case bias over smoothness classes like H\"older balls
of functions. Lepski's method can be understood as a multiple testing procedure
where the hypothesis $H_0(k):\;g|_{U_k}\equiv \theta$ that $g$ is constant on
$U_k$ is tested against the alternative of significant deviations. Always
assuming that $H_0(0)$ is true, we test sequentially whether $H_0(k+1)$ is
acceptable provided that the hypotheses $H_0(\ell)$ have been accepted for all
$\ell\le k$. Once the test of an $H_0(k+1)$ is rejected, we select the base
estimator $\tilde\theta_k$ corresponding to the last accepted hypothesis. The
main point is thus to properly define the single significance tests for
$H_0(k+1)$. Lepski's method accepts $H_0(k+1)$ if
$\abs{\tilde\theta_{k+1}-\tilde\theta_\ell}\le z_\ell^{(k+1)}$ holds for all
$\ell\le k$ with suitable critical values $z_\ell^{(k+1)}>0$. The wide
applicability and success of Lepski's method is also due to this very simple
and intuitive test statistics.

In our nonlinear estimation case it turns out that tests for $H_0(k+1)$ based on the differences of base estimators are often not optimal. To understand this fact, let us consider a toy model of two neighbourhoods $U_1\subset U_2$ with a piecewise constant median regression function $g$ equal to $\mu_1$ on $U_1$ and to $\mu_2$ on $U_2\setminus U_1$. The procedure therefore reduces to a simple two-sample location test between the observations in $U_1$ and in $U_2\setminus U_1$. We proceed by considering abstractly a two-sample location test where the first sample $Y_1,\ldots,Y_n$ is i.i.d. with density $f_1(x)=\frac1{2\sigma} \exp(-\abs{x-\mu_1}/\sigma)$ and the second independent sample $Y_{n+1},\ldots,Y_{2n}$ is i.i.d. with density $f_2(x)=\frac1{2\sigma} \exp(-\abs{x-\mu_2}/\sigma)$. Our goal is to test $H_0: \mu_1=\mu_2$ for known $\sigma>0$. Given the Laplace distribution, we follow Lepski's idea and put
$\tilde m_1=\med(Y_i,\,i=1,\ldots,n)$, the median over the first sample, and $\tilde m_2=\med(Y_i,\,i=1,\ldots,2n)$, the median over both samples. Then the test rejects if $T_L:=2\abs{\tilde m_1-\tilde m_2}>z$ holds for appropriate $z>0$.
A more classical approach, though, relies on a likelihood ratio (LR) test or on a Wald-type test using the maximum likelihood estimator for $\mu_1-\mu_2$. Since the LR test is not as simple, we focus on the Wald-test statistic which is given by the difference of
the medians over the two samples. Hence, we reject $H_0$ if $T_W:=\abs{\med(Y_i,\,i=1,\ldots,n)-\med(Y_i,\,i=n+1,\ldots,2n)}>z$ holds for appropriate $z>0$.
The following asymptotic result for the two test statistics is proved in Section \ref{SecApp1}.

\begin{proposition}\label{PropTest}
Let $f:\R\to\R^+$ be a symmetric and continuous density with $f(0)>0$ and let $Y_1,\ldots,Y_n\sim f$, $Y_{n+1},\ldots,Y_{2n}\sim f(\cdot-\Delta)$ with $\Delta>0$ be independently distributed. Then with $F$ denoting the cumulative distribution function of $f$ we obtain for $n\to\infty$
\begin{align*}
 &\sqrt{n}\Big(\med(Y_i,\,i=n+1,\ldots,2n)-\med(Y_i,\,i=1,\ldots,n)-\Delta\Big) \Rightarrow N(0,\sigma_W^2)\\
&\quad\text{ with } \sigma_W^2=\frac{1}{2f^2(0)},\\
&\sqrt{n}\Big(2\Big(\med(Y_i,\,i=1,\ldots,2n)-\med(Y_i,\,i=1,\ldots,n)\Big)-\Delta\Big) \Rightarrow N(0,\sigma_L^2)\\
&\quad\text{ with } \sigma_L^2=
\frac{2F(\Delta/2)(1-F(\Delta/2))}{f^2(\Delta/2)}+\frac{1}{f^2(0)}-\frac{2(1-F(\Delta/2))}{f(0)f(\Delta/2)}.
\end{align*}
In particular, for $\Delta=0$ we have $\sigma_L^2=\sigma_W^2$ and for $\Delta\to 0$ we have the order $\sigma_L^2=\sigma_W^2(1+2\Delta f(0)+O(\Delta^2f(0)))$, provided $f$ is Lipschitz continuous at zero.
\end{proposition}

Putting $\Delta=\abs{\mu_1-\mu_2}$ this result shows that under $H_0$, i.e. $\Delta=0$, the test statistics $T_L$ and $T_W$ are asymptotically identically distributed, whereas $T_L$ has a larger asymptotic variance under any alternative $\Delta>0$ than $T_W$. In the above Laplace model with densities $f_1,f_2$ this deterioration is only negligible if the signal-to-noise ratio satisfies $\abs{\mu_1-\mu_2}/\sigma\ll 1$. This is exactly what we see in simulations, see e.g. Example 1 in Section \ref{SecSimul} below. Since the Laplace model
is Hellinger differentiable, the Wald-type test is (locally) asymptotically efficient for $n\to\infty$ as is the LR test, see e.g. \cit{vanderVaart}. Strictly speaking, when considering local alternatives for fixed $\sigma>0$ and $n\to\infty$, i.e. $\abs{\mu_1-\mu_2}=O(n^{-1/2})$, then
the deterioration in using $T_L$ becomes also negligible. From a practical perspective, these local asymptotics are often not adequate, e.g. in image denoising, where we face relatively large signal differences $\Delta$ at borders between objects and do not dispose of a very large number $n$ of observed pixels.

More generally, two-sample location tests can naturally be based on the
difference of the in-sample location estimators. In consequence, we proceed
differently in testing the hypotheses $H_0(k+1)$ of homogeneity: When the
hypotheses $H_0(\ell)$ for $\ell\le k$ have been accepted, we ask whether the
observations $Y_i$ in the new points $x_i\in U_{k+1}\setminus U_k$ are
homogeneous with those in $U_\ell$ for $\ell\le k$. This means that our tests
reject if the empirical location in the additional data
\[\tilde\theta_{(k+1)\setminus k}:=m(Y_i,\,x_i\in U_{k+1}\setminus U_k)\]
satisfies with certain critical values $z_\ell^{(k+1)}>0$:
\[ \exists\ell\le k:\:\abs{\tilde\theta_{(k+1)\setminus k}-\tilde\theta_\ell}>z_\ell^{(k+1)}.\]
As in Lepski's method, it is necessary to perform the testing for all $\ell\le
k$ and not only with $\ell=k$ to avoid that the signal slowly drifts away as
the neighbourhoods grow. In most cases, though, $H_0(k+1)$ will be rejected
because the new piece $\tilde\theta_{(k+1)\setminus k}$ is not in line with
$\tilde\theta_k$: due to the smaller variance of $\tilde\theta_k$ compared to
$\tilde\theta_\ell$, $\ell< k$, this last test is the most powerful. It is then
interesting to observe that for linear $m$ the test statistic
$\tilde\theta_{(k+1)\setminus k}-\tilde\theta_k$ is just a multiple of
$\tilde\theta_{k+1}-\tilde\theta_k$. Consequently, for mean regression with
linear base estimators our method will not differ much from Lepski's standard
method, whereas the general nonlinear M-estimators are treated in a
significantly different way, note also the numerical results in Section
\ref{SecSimul}.

Observe that our approach breaks an ubiquitous paradigm in modern statistics and learning theory (see e.g. \cit{Massart} for model selection or \cit{Tsybakov} for aggregation): we select the best base learner among $(\tilde\theta_k)$ in a data-driven way not only based on the estimator values themselves, but additionally on the statistics $(\tilde\theta_{(k+1)\setminus k})$. Not only in the abstract modeling above, but also in implementations this idea turns out to be very advantageous for nonlinear estimators.

\subsection{The algorithm}\label{SecAlgo}

We want to select the best estimator among the family
$\{\tilde\theta_k\,|\,k=0,\ldots,K\}$. Considering the law
$\PP_0$ generated by the no-bias setting $g\equiv 0$, we
introduce the stochastic error levels
\begin{equation}\label{EqStochErr}
s_j:=\E_0[\abs{\tilde\theta_j}^r]^{1/r},\quad
s_{kj}:=\E_0[\abs{\tilde\theta_{(k+1)\setminus
k}-\tilde\theta_j}^r]^{1/r}.
\end{equation}
We apply the following sequential procedure for prescribed critical values
$(z_j)_{j=0,\ldots,K-1}$ and set $z_K:=1$:
{\tt
\begin{itemize}
\item initialize $k:=0$;
\item repeat\\
\hspace*{3mm} if for all $j=0,\ldots,k$
\[\abs{\tilde\theta_{(k+1)\setminus
k}-\tilde\theta_j}\le z_j s_{kj}+ z_{k+1} s_{k+1}
\]
\hspace*{6mm} then increase $k$\\
\hspace*{6mm} else stop\\
until $k=K$;
\item put $\hat k:=k$ and $\hat\theta:=\tilde\theta_{\hat k}$.
\end{itemize}}
This algorithm to determine $\hat k$ can be cast in one
formula:
\[ \hat k:=\inf\Big\{k\ge 0\,\Big|\,\exists j\le k:\:
\abs{\tilde\theta_{(k+1)\setminus
k}-\tilde\theta_j}> z_j s_{kj}+ z_{k+1}
s_{k+1}\Big\}\wedge K.
\]

\section{Error analysis}\label{SecError}

\subsection{Propagation and stopping late}

We need a very natural property of the $M$-estimator.

\begin{assumption}\label{AssMean}
 The location estimator in \eqref{eqmdef} satisfies for any set $S$ and any partition $S=\bigcup_j S_j$ with pairwise disjoint sets $S_j$:
\[ \textstyle\min_j m(Y_i,\,x_i\in S_j) \le m(Y_i,\,x_i\in S)\le \max_j m(Y_i,\,x_i\in S_j).\]
\end{assumption}

\begin{lemma}\label{LemMeanconvex}
 If the function $\rho$ is strictly convex, then Assumption \ref{AssMean} is satisfied.
\end{lemma}

\begin{proof}
 Let us write $m_T$ as short-hand for $m(Y_i,\,x_i\in T)$, $T\subset{\cal X}$. Denoting by $\rho'_+,\rho'_-$ the right- and left-handed derivatives of the convex function $\rho$, the functions $\rho'_+$, $\rho'_-$ are strictly increasing with $\rho'_+(x)< \rho'_-(y)\le \rho'_+(y)$ for all $x<y$ and
\[ \sum_{x_i\in T}\rho'_-(Y_i-m_T)\le 0,\quad \sum_{x_i\in T}\rho'_+(Y_i-m_T)\ge 0.
\]
If $m_S<m_{S_j}$ were true for all $j$, then
\[ \sum_{x_i\in S}\rho'_-(Y_i-m_S)>
\sum_j\sum_{x_i\in S_j}\rho'_+(Y_i-m_{S_j})\ge 0,
\]
which contradicts the minimizing property of $m_S$. Hence, $m_S\ge\min_j m_{S_j}$ holds and a symmetric argument shows $m_s\le \max_jm_{S_j}$.
\end{proof}

\begin{remark}
 If $\rho$ is not strictly convex, then we usually impose additional conditions to define $m$ uniquely. For any reasonable specific choice Assumption \ref{AssMean} should be satisfied. In particular, this is true for the sample median where we take for an even number $N$ of data points $Y_i$ the mean $(Y_{(N/2)}+Y_{(1+N/2)})/2$ of the order statistics.
\end{remark}

\begin{comment}
\begin{lemma}\label{LemMed}
Let $x_1,\ldots,x_{2m+1}$ be arbitrary real numbers and suppose that
$0\le m_0< \cdots< m_k=m$. Then
\begin{align*}
&\abs{\med(x_i,\,i=1,\ldots,2m+1)-\med(x_i,\,i=1,\ldots,2m_0+1)}\\
&\quad \le \max_{j=1,\ldots,k}
\abs{\med(x_i,\,i=2(m_{j-1}+1),\ldots,2m_j+1)
-\med(x_i,\,i=1,\ldots,2m_0+1)}.
\end{align*}
\end{lemma}
\begin{proof}
By homogeneity of the median we can assume without loss of
generality that $\med(x_i,\,i=1,\ldots,2m_0+1)=0$. Then the
assertion reduces to the fact that the median over a set is always
bounded in absolute value by the maximal median over subsets from a
partition of this set. (\MR clear...?)
\end{proof}
\end{comment}

\begin{proposition}\label{PropLateErr}
Grant Assumption \ref{AssMean}. Then we have for any $k=0,\ldots,K-1$
\[
\abs{\hat\theta-\tilde\theta_k}{\bf 1}(\hat k>k)\le
\max_{j=k+1,\ldots,K-1} \big(z_k s_{jk}+ z_{j+1} s_{j+1}\big).
\]
\end{proposition}

\begin{remark}
This error propagation result is true '$\omega$-wise', that is, it does not depend on
the noise realisation. It is built into the construction of the selection
procedure. An analogous result holds for Lepski's original procedure \cite{Lepski,Lepskietal}.
\end{remark}

\begin{proof}
From Assumption \ref{AssMean} we infer for $\ell>k$
\[ \abs{\tilde\theta_\ell-\tilde\theta_k}\le \max_{k+1\le j\le
\ell}\abs{\tilde\theta_{j\setminus(j-1)}-\tilde\theta_k}.
\]
We therefore obtain on the event $\{\hat k>k\}$ by construction
\begin{align*}
\abs{\hat\theta-\tilde\theta_k}&\le \max_{j=k+1,\ldots,\hat k}
\abs{\tilde\theta_{j\setminus(j-1)}-\tilde\theta_k}\le
\max_{j=k+1,\ldots,K-1} \big( z_k s_{jk}+ z_{j+1} s_{j+1} \big).
\end{align*}

\end{proof}

\begin{example}
For geometrically decreasing stochastic error levels $s_k$ in \eqref{EqStochErr}, in particular for the median
filter from Example \ref{ex1} with bandwidths $h_k=h_0q^k$,  we have $s_{jk}\lesssim s_k$
for $j>k$, where $A\lesssim B$ means $A={\cal O}(B)$ in the ${\cal O}$-notation. The late stopping error is of order $z_k^rs_k^r$,
provided the critical values $(z_k)$ are non-increasing. This will
imply that the error due to stopping later than some optimal
$k^\ast$ is increased by at most the order of $z_{k^\ast}^r$:
\[ \E_\theta[\abs{\hat\theta-\theta}^r{\bf 1}(\hat k>k^\ast)]
\lesssim \E_\theta[\abs{\tilde\theta_{k^\ast}-\theta}^r]+ z_{k^\ast}^rs_{k^\ast}^r
\le (1+z_{k^\ast}^r)\E_\theta[\abs{\tilde\theta_{k^\ast}-\theta}^r].
\]
\end{example}

\subsection{Critical values and stopping early}

As the preceding analysis shows, small critical values $(z_k)$ lead to small errors caused by stopping late. On the other hand, the $(z_k)$ should not be too small in order to control the error of stopping early.
To this end, we shall require a condition on the critical values $(z_k)$ in the
no-bias situation under $\PP_0$, that is for constant $g\equiv 0$. In
fact, we face a multiple testing problem, but with an
estimation-type loss function. For some confidence parameter
$\alpha>0$ we select $z_k>0$, $k=0,\ldots,K-1$, such that the condition
\begin{equation}\label{eqhypzkgen}
 \sum_{j=0}^{K-1}\E_0\Big[\abs{\tilde\theta_j}^r {\bf
1}\big(\exists \ell\le j:\; \abs{\tilde\theta_{(j+1)\setminus j}-\tilde\theta_\ell}>z_\ell s_{j\ell}\big)\Big]
\le\alpha s_K^r
\end{equation}
is satisfied. In order to obtain a unique prescription for each $z_k$ that equilibrates the errors for different stopping times of the algorithm, we can select the $(z_k)$ sequentially. We choose $z_0$ such that
\[ \sum_{j=0}^{K-1}\E_0\Big[\abs{\tilde\theta_j}^r {\bf
1}\big(\abs{\tilde\theta_{(j+1)\setminus j}-\tilde\theta_0}>z_0s_{j0}\big)\Big]
\le\tfrac{\alpha}{K} s_K^r
\]
and then each $z_k$ for given $z_0,\ldots,z_{k-1}$ such that
\begin{equation}\label{eqhypzk}
\sum_{j=k}^{K-1}\E_0\Big[\abs{\tilde\theta_j}^r {\bf
1}\Big(\abs{\tilde\theta_{(j+1)\setminus j}-\tilde\theta_k}>z_ks_{jk}\text{, }\forall \ell<k:\;
\abs{\tilde\theta_{(j+1)\setminus j}-\tilde\theta_\ell}\le z_\ell s_{j\ell}\Big)\Big]
\le\tfrac{\alpha}{K} s_K^r.
\end{equation}
To determine the $(z_k)$ in practice, we simulate in Monte Carlo iterations the pure noise
case $g\equiv0$ and calculate for each $k$ the error when the algorithm stops before the
(theoretically optimal) index $K$ due to a rejected test involving $z_k$. The critical values are determined such that this error is a fraction of the oracle estimation error $s_K^r$. For this calibration step the original
algorithm of Section \ref{SecAlgo} is taken, only modified by using $z_js_{kj}$ instead of $z_js_{kj}+z_{k+1}s_{k+1}$
in the testing parts.

The selection rule for the critical values in Lepski's procedure is the focus in the work by \cit{SpokVial}. Their idea is to transfer properties from the no-bias situation to the general nonparametric specification by bounding the likelihood between the two observation models. This approach, the so-called small modeling bias condition, could be applied here as well and will give similar results. On a practical level, the difference is that \cit{SpokVial} enlarge the moment from $r$ to $2r$ in the calibration step, while we add the term $z_{k+1}s_{k+1}$ to the testing values $z_js_{kj}$ from the calibration. In the asymptotic analysis, however, the method by \cit{SpokVial} costs us some power in the logarithmic factor and we would thus not attain optimal rates over H\"older balls, cf. Section \ref{SecAsymp}. Moreover, for robustness reasons, we do not want to require higher moment bounds for the error variables and the likelihood.

\begin{definition}
Given the regression function $g$, introduce its variation on $U_k$
\[\V_k(g):=\sup_{y_1,y_2\in U_k}\abs{g(y_1)-g(y_2)}\]
and consider the oracle-type
index
\[ k^\ast:=\min \{k=0,\ldots,K-1\,|\, \V_{k+1}(g)> z_{k+1}s_{k+1}\}\wedge K.
\]
\end{definition}

This definition implies that for all $k\le k^\ast$
the maximal bias $V_k(g)$ of $\tilde\theta_k$ is less than its stochastic error level $s_k$ from \eqref{EqStochErr} times the critical value $z_k$. The
next result, when specialised to $k=k^\ast$, means intuitively that
the error due to stopping before $k^\ast$ can be bounded in terms of
the stochastic error of $\tilde\theta_{k^\ast}$, involving the
critical value $z_{k^\ast}$ as a factor. Let us also mention here that the rationale for the choice $z_K=1$ in the algorithm of Section \ref{SecAlgo} is to equilibrate maximal bias and stochastic error at step $k=K-1$.

\begin{proposition}\label{PropEarlyrr}
We have for any $k=0,\ldots,k^\ast$
\[
\E\big[\abs{\hat\theta-\tilde\theta_k}^r{\bf 1}(\hat k<k)\big]\le
(3^{r-1}\vee 1)(z_k^r+1+\alpha)s_k^r.
\]
\end{proposition}

\begin{proof}
We shall write $\hat k(g)$, $\tilde\theta_k(g)$ etc. to indicate
that $\hat k$, $\tilde\theta_k$ etc. depend on the underlying
regression function $g$. We shall need the inequality
\begin{equation}\label{eqEstg0}
 \abs{\tilde\theta_j(g)-\tilde\theta_k(g)}\le \abs{\tilde\theta_j(0)-\tilde\theta_k(0)}+V_k(g)\text { for } j<k
\end{equation}
which follows from
\begin{align*}
\tilde\theta_j(g)-\tilde\theta_k(g) &=m(g(x_i)+\eps_i,\,x_i\in U_j)-m(g(x_i)+\eps_i,\,x_i\in U_k)\\
&\le m(\eps_i,\,x_i\in U_j)+\sup_{x\in U_j}g(x)-m(\eps_i,\,x_i\in U_k)-\inf_{x\in U_k}g(x)\\
&\le \tilde\theta_j(0)-\tilde\theta_k(0)+V_k(g)
\end{align*}
and by a symmetric argument for $\tilde\theta_k(g)-\tilde\theta_j(g)$.

By definition of $k^\ast$ and using the
condition on the $(z_k)$ as well as \eqref{eqEstg0} for $\tilde\theta_j$ and $\tilde\theta_{(j+1)\setminus j}$, we obtain for all $k\le k^\ast$
\begin{align*}
&\E\big[\abs{\hat\theta(g)-\tilde\theta_k(g)}^r{\bf 1}(\hat k(g)<k)\big]\\
&=\sum_{j=0}^{k-1}\E\big[\abs{\tilde\theta_j(g)-\tilde\theta_k(g)}^r{\bf
1}(\hat k(g)=j)\big]\\
&\le \sum_{j=0}^{k-1}\E\big[(\V_k(g) +
\abs{\tilde\theta_j(0)}+\abs{\tilde\theta_k(0)})^r{\bf
1}(\hat k(g)=j)\big]\\
&\le (3^{r-1}\vee
1)\Big(\V_k(g)^r+\E[\abs{\tilde\theta_k(0)}^r]+\\
&\quad+ \sum_{j=0}^{k-1}\E\big[\abs{\tilde\theta_j(0)}^r {\bf
1}\big(\exists \ell\le j:\:\abs{\tilde\theta_{(j+1)\setminus j}(g)-\tilde\theta_\ell(g)}> z_\ell s_{j\ell}+ z_{j+1} s_{j+1}\big)\big]\Big)\\
&\le (3^{r-1}\vee
1)\Big(z_k^r s_k^r+s_k^r+\\
&\quad +\sum_{j=0}^{k-1}\E\big[\abs{\tilde\theta_j(0)}^r {\bf
1}\big(\exists \ell\le j:\:\abs{\tilde\theta_{(j+1)\setminus j}(0)-\tilde\theta_\ell(0)}+V_{j+1}(g)> z_\ell s_{j\ell}+ z_{j+1} s_{j+1}\big)\big]\Big)\\
%&\le (3^{r-1}\vee 1)\Big((z_k^r+1)s_k^r+
%\sum_{j=0}^{k-1}\E\big[\abs{\tilde\theta_j(0)}^r {\bf
%1}(\exists l\le j:\:\abs{\med_{(j+1)\setminus j}(0)-\tilde\theta_l(0)}> z_l s_{jl}+ z_{j+1} s_{j+1}-\V_{j+1}(f))\big]\Big)\\
%&\le (3^{r-1}\vee 1)\Big(z_k^r s_k^r+s_k^r+
%\sum_{j=0}^{k-1}\sum_{\ell=0}^j\E\big[\abs{\tilde\theta_j(0)}^r {\bf
%1}(\abs{\tilde\theta_{(j+1)\setminus j}(0)-\tilde\theta_\ell(0)}> z_\ell s_{j\ell})\big]\Big)\\
&\le (3^{r-1}\vee 1)\Big(z_k^rs_k^r+s_k^r+\alpha s_K^r\Big).
\end{align*}
The result follows from the isotonic decay of $(s_k)$.
\end{proof}

\begin{comment}
\begin{proposition}\label{PropEarlyrr}
We have for any $k=0,\ldots,k^\ast$
\[
\E\big[\abs{\hat\theta}^r{\bf 1}(\hat k<k)\big]\le
(2^{r-1}\vee 1)(z_k^r+\alpha)s_k^r.
\]
\end{proposition}

\begin{proof}
We shall write $\hat k(f)$, $\tilde\theta_k(f)$ etc. to indicate
that $\hat k$, $\tilde\theta_k$ etc. depend on the underlying
regression function $f$.
By definition of $k^\ast$ and using the
condition on the $(z_k)$, we obtain for all $k\le k^\ast$
\begin{align*}
&\E\big[\abs{\hat\theta(f)}^r{\bf 1}(\hat k(f)<k)\big]\\
&=\sum_{j=0}^{k-1}\E\big[\abs{\tilde\theta_j(f)}^r{\bf
1}(\hat k(f)=j)\big]\\
&\le \sum_{j=0}^{k-1}\E\big[(\V_k(f) +
\abs{\tilde\theta_j(0)})^r{\bf
1}(\hat k(f)=j)\big]\\
&\le (2^{r-1}\vee
1)\Big(\V_k(f)^r\\
&\quad+ \sum_{j=0}^{k-1}\E\big[\abs{\tilde\theta_j(0)}^r {\bf
1}\big(\exists \ell\le j:\:\abs{\tilde\theta_{(j+1)\setminus j}(Y(f))-\tilde\theta_\ell(f)}> z_\ell s_{j\ell}+ z_{j+1} s_{j+1}\big)\big]\Big)\\
%&\le (3^{r-1}\vee 1)\Big((z_k^r+1)s_k^r+
%\sum_{j=0}^{k-1}\E\big[\abs{\tilde\theta_j(0)}^r {\bf
%1}(\exists l\le j:\:\abs{\med_{(j+1)\setminus j}(0)-\tilde\theta_l(0)}> z_l s_{jl}+ z_{j+1} s_{j+1}-\V_{j+1}(f))\big]\Big)\\
&\le (2^{r-1}\vee 1)\Big(z_k^rs_k^r+
\sum_{j=0}^{k-1}\sum_{\ell=0}^j\E\big[\abs{\tilde\theta_j(0)}^r {\bf
1}(\abs{\tilde\theta_{(j+1)\setminus j}(0)-\tilde\theta_\ell(0)}> z_\ell s_{j\ell})\big]\Big)\\
&\le (2^{r-1}\vee 1)\Big(z_k^rs_k^r+\alpha s_K^r\Big).
\end{align*}
The result follows from the isotonic decay of $(s_k)$.
\end{proof}
\end{comment}

\subsection{Total risk bound}

\begin{theorem}\label{thmrisk}
Assume that $(z_ks_k)$ is non-increasing in $k$. Then under Assumption \ref{AssMean} the following
excess risk estimate holds for all $k\le k^\ast$:
\[ \E[\abs{\hat\theta-\tilde\theta_k}^r]\le
(3^{r-1}\vee 1)\Big( (2z_k^r+1+\alpha)s_k^r +
z_k^r\max_{j=k+1,\ldots,K-1}s_{jk}^r\Big).
\]
\end{theorem}

\begin{proof}
For the late-stopping error Proposition \ref{PropLateErr} and the
decay of $(z_ks_k)$ give
\[\abs{\hat\theta-\tilde\theta_k}^r{\bf 1}(\hat k>k)\le  (2^{r-1}\vee 1)
\max_{j>k}(z_k^r s_{jk}^r+ z_{j+1}^r s_{j+1}^r)\le (2^{r-1}\vee 1)
z_k^r \big(s_k^r+\max_{j>k}s_{jk}^r\big).
\]
Add the early-stopping error from Proposition \ref{PropEarlyrr}.
\end{proof}

\begin{comment}
\begin{theorem}\label{thmrisk}
Assume that $z_ks_k$ is non-increasing in $k$. Then the following
risk estimate holds:
\begin{align*}
\E[\abs{\hat\theta}^r]&\le
(2^{r-1}\vee 1)\min_{k\le k^\ast} \Big( (2^r\vee 2)z_k^r+1+\alpha)s_k^r +
z_k^r(2^{r-1}\vee 1)\max_{j=k+1,\ldots,K-1}s_{jk}^r\Big).
\end{align*}
\end{theorem}

\begin{proof}
For the late-stopping error Proposition \ref{PropLateErr} and the
assumption give
\[\abs{\hat\theta-\tilde\theta_k}^r{\bf 1}(\hat k>k)\le  (2^{r-1}\vee 1)
\max_{j>k}(z_k^r s_{jk}^r+ z_{j+1}^r s_{j+1}^r)\le (2^{r-1}\vee 1)
z_k^r \big(s_k^r+\max_{j>k}s_{jk}^r\big).
\]
Then use $\abs{\hat\theta}^r\le(2^{r-1}\vee 1)(\abs{\hat\theta-\tilde\theta_k}^r+\abs{\tilde\theta_k}^r)$
to obtain
\[ \E\big[\abs{\hat\theta}^r{\bf 1}(\hat k>k)\big]\le
 (2^{r-1}\vee 1)^2
z_k^r \big(s_k^r+\max_{j>k}s_{jk}^r\big)+(2^{r-1}\vee 1)s_k^r.
\]
Finally, add the early-stopping error from Proposition \ref{PropEarlyrr}.
\end{proof}
\end{comment}

\begin{example}[continued]
For geometrically increasing bandwidths $(h_k)$ we obtain
$s_{jk}\lesssim s_k$ for $j>k$ and thus
\[\E[\abs{\hat\theta-\tilde\theta_{k^\ast}}^r]\lesssim
(\alpha+z_{k^\ast}^r)s_{k^\ast}^r.
\]
The factor $\alpha+z_{k^\ast}^r$ is the term we pay for
adaptation.
\end{example}

\section{Asymptotic risk}\label{SecAsymp}

\subsection{General result}

We shall derive convergence rates for $n\to\infty$ of the critical
values $(z_k)$. All quantities in the procedure may depend on $n$, but
we still write $U_k$, $K$ and $z_k$ instead of $U_k(n)$, $K(n)$, $z_k(n)$.
The notation $A\lesssim B$ will always mean  $A(n)\le cB(n)$ with some $c>0$ independent of $n$ and $A\thicksim B$ is short for $A\lesssim B$ and $B\lesssim A$.
We work under the following assumption whose validity under mild conditions will be derived
in the next subsection.

\begin{assumption}\label{AssAsymp}\mbox{}
\begin{enumerate}
\item The cardinalities $N_k$ of the neighbourhoods $U_k$ grow with geometric order:
\[ q_1N_k\le N_{k+1}\le
q_2N_k\quad\text{ for all $k=0,\ldots K-1$}
\]
for some fixed $q_2\ge q_1>1$ and with $N_1/\log(N_K)\to\infty$, $N_K\thicksim n$ as
$n\to\infty$.

\item For all sufficiently large $N$ we have
\[ \E[\abs{m(\eps_i,\,i=1,\ldots,N)}^{r}]^{1/r}\thicksim
\E[\abs{m(\eps_i,\,i=1,\ldots,N)}^{2r}]^{1/2r}\thicksim N^{-1/2}.
\]

\item For all $\tau_N\to\infty$ with $\tau_N N^{-1/2}\to 0$ a moderate deviations bound applies:
there is some $c>0$ such that
\[ \limsup_{N\to\infty}e^{c\tau_N^2}\PP\big(N^{1/2}\abs{m(\eps_i,\,i=1,\ldots,N)}>\tau_N\big)<\infty.
\]

\end{enumerate}
\end{assumption}

The following asymptotic bounds follow directly from the definitions:

\begin{lemma}\label{Lemskj}
Assumption \ref{AssAsymp}(b) implies $s_j\thicksim
N_j^{-1/2}$ and $N_j^{-1/2}\wedge(N_{k+1}-N_k)^{-1/2}\lesssim
s_{kj}\lesssim N_j^{-1/2}\vee(N_{k+1}-N_k)^{-1/2}$. Assumption
\ref{AssAsymp}(a) then yields for $k\ge j$
\[ s_j\thicksim s_{kj}\thicksim
N_j^{-1/2}.
\]
\end{lemma}

Under Assumption \ref{AssAsymp} critical values of the same order as in the Gaussian case suffice.

\begin{proposition}\label{propasympzk}
Grant Assumption \ref{AssAsymp} and suppose $\alpha\in (0,1)$. We
can choose
\[z_k^2=\zeta\big(2r\log(s_k/s_K)+\log(\alpha^{-1})+\log(K)\big),\quad k=0,\ldots,K-1,
\]
with $\zeta>0$ a sufficiently large constant in order to satisfy Condition \eqref{eqhypzk}. For $K\thicksim \log
n$ this yields asymptotically $z_k\thicksim\sqrt{\log n}$.
\end{proposition}

Note that the chosen critical values $z_k$ are decreasing in $k$, which has
the desirable effect that we do not permit stopping at an early
stage with the same probability as stopping at higher indices $k$.
Moreover, this guarantees that $z_ks_k$ is non-increasing in $k$,
the hypothesis in Theorem \ref{thmrisk}. From Theorem \ref{thmrisk} we therefore obtain the following
asymptotic risk bound.

\begin{corollary}\label{corasymprisk}
 Grant Assumptions \ref{AssMean} and \ref{AssAsymp} and let $K\thicksim \log n$. Choosing the critical values as in Proposition \ref{propasympzk}
gives
\[ \E[\abs{\hat\theta-\theta}^r]\lesssim (\log n)^{r/2}\E[\abs{\tilde\theta_{k^\ast}-\theta}^r].\]
\end{corollary}

\begin{example}[continued]
Let us specify to $s$-H\"older continuous $g:[0,1]\to\R$,
equidistant design and kernel estimators with geometrically
increasing bandwidths $h_k=h_0q^k$, $K\thicksim\log(n)$. Then we can choose
$z_k\thicksim \sqrt{\log(n)}$ and the index $k^\ast$ satisfies
$\V_{k^\ast}(f)^2\thicksim h_{k^\ast}^{2s}\thicksim
(nh_{k^\ast})^{-1}\log(n)$, that is $h_{k^\ast}\thicksim
(\log(n)/n)^{1/(2s+1)}$ and $z_{k^\ast}s_{k^\ast}\thicksim
(\log(n)/n)^{s/(2s+1)}$. This is the classical minimax rate for
pointwise adaptive estimation in the one-dimensional $s$-H\"older
continuous case, see \cit{Lepskietal} for the Gaussian case. Here, we have derived the same rate for pointwise adaptive $M$-estimation under very weak conditions on the error distribution, compare the discussion on specific models below.
Let us also mention that \cit{Truong} obtains the same rate result, but
without logarithmic factor, for the non-adaptive median regression
case.
\end{example}

\begin{proof}[Proof of Proposition \ref{propasympzk}]
Let $j\ge k$. For $n$ sufficiently large Assumption \ref{AssAsymp}(c) together with the asymptotics $z_ks_{jk}\lesssim (\log(N_K)N_k^{-1})^{1/2}\to 0$
(using Assumption \ref{AssAsymp}(a,b) and Lemma \ref{Lemskj}) yields
\begin{align*}
&\PP_0(\abs{\tilde\theta_{(j+1)\setminus j}-\tilde\theta_k}>z_k
s_{jk})\\
&\le \PP_0(\abs{\tilde\theta_{(j+1)\setminus j}}>z_k
s_{jk}/2)+\PP_0(\abs{\tilde\theta_k}>z_k
s_{jk}/2)\\
&\lesssim \exp(-c z_k^2s_{jk}^2(N_{j+1}-N_j)/4)+\exp(-c
z_k^2s_{jk}^2N_k/4).
\end{align*}
By Lemma \ref{Lemskj} there is another constant
$c'>0$ such that for large $z_k$
\[ \PP_0(\abs{\tilde\theta_{(j+1)\setminus j}-\tilde\theta_k}>z_k
s_{jk})\lesssim \exp(-c'z_k^2).
\]
Our choice of $z_k$ with $\zeta$ sufficiently large guarantees
$\exp(-c'z_k^2/2)=o(\alpha(s_K/s_k)^r K^{-2})$ for large
$K$. We therefore more than satisfy \eqref{eqhypzkgen} and the construction in \eqref{eqhypzk}
provided $n$ is sufficiently large:
\begin{align*}
\sum_{j=k}^{K-1} & \E_0\big[\abs{\tilde\theta_j}^r {\bf
1}(\abs{\tilde\theta_{(j+1)\setminus j}-\tilde\theta_k}>z_ks_{jk})\big]\\
&\le \sum_{j=k}^{K-1}\E_0\big[\abs{\tilde\theta_j}^{2r}]^{1/2} \PP_0(\abs{\tilde\theta_{(j+1)\setminus j}-\tilde\theta_k}>z_ks_{jk})^{1/2}\\
&\lesssim \sum_{j=k}^{K-1}s_j^r\exp(-c'z_k^2/2)\\
&=o\big( (K-k)s_k^r\alpha(s_K/s_k)^{r}K^{-2}\big)\\
&= o\Big(\frac{\alpha s_K^r}{K}\Big).
\end{align*}
For $K\thicksim \log N$ we obtain $\log(N_K/N_k)\le (K-k)\log q_2\lesssim \log N$ and thus $z_k^2\thicksim\log n$.
\end{proof}

\subsection{Specific models}

The preceding asymptotic analysis was based on Assumption \ref{AssAsymp} where part (a) can be ensured by construction whereas parts (b) and (c) depend on
the noise model and the choice of M-estimator. The most severe restriction will usually be the moderate deviation property of Assumption \ref{AssAsymp}(c). In the case where the law of the error variable $\eps_i$ is absolutely continuous, this property holds by Corollary 2.1 in \cit{Arcones} under the following conditions:
\begin{enumerate}
 \item $\E[\rho(\eps_i+h)-\rho(\eps_i)]=Vh^2+o(h^2)$ for some $V>0$ and $\abs{h}\to 0$;
\item $\rho$ is Lebesgue-almost everywhere differentiable with derivative $\rho'$;
\item there are $\lambda,\delta>0$ such that $\E[\exp(\lambda\abs{\rho'(\eps_i)})]$ and
$\E[\exp(\lambda\sup_{\abs{h}\le\delta}\abs{\rho(\eps_i+h)-\rho(\eps_i)-h\rho'(\eps_i)}/h)]$ are finite.
\end{enumerate}
For mean regression $\rho(x)=x^2$ we have $V=1$ and $\rho'(\eps_i)=2\eps_i$ such that a finite exponential moment for $\eps_i$ is required. For median regression the result applies with $V=f_\eps(0)/2$ and $\rho'(\eps_i)=\sgn(\eps_i)$ and because of $\abs{\abs{\eps_i+h}-\abs{\eps_i}-h\sgn(\eps_i)}\le 2h$ no moment bound is required. The same is true for any robust statistic with bounded influence function, in particular for the Huber estimator and general quantile estimators.
\cit{Arcones} discusses that an exponential tail estimate for $\rho'(\eps_i)$ is also necessary to obtain a moderate deviation bound, which might be a serious drawback when using Lepski's method with linear non-robust estimators.

For the median the requirements are not difficult to verify directly. Assumption \ref{AssAsymp}(b) is for example established by \cit{ChuHot}, who show that for $f_\eps$ continuously differentiable around zero, $f_\eps(0)>0$, $r\in\N$ and $Z\sim N(0,1)$:
\[ \lim_{N\to\infty}N^r\E[\med(\eps_1,\ldots,\eps_N)^{2r}]=(2f_\eps(0))^{-r}\E[Z^{2r}].
\]
Using a coupling result, we can establish Assumption \ref{AssAsymp}(b,c) under even more general conditions, see Section \ref{SecApp2} for a proof:

\begin{proposition}\label{PropMoments}
Assume that the $\eps_i$ have a Lebesgue density $f_\eps$ which is Lipschitz continuous at zero and satisfies
$\int_{-\infty}^0 f_\eps(x)\,dx=1/2$, $f_\eps(0)>0$, $\E[\abs{\eps_i}^r]<\infty$.
Noting $\med(\eps):=\med(\eps_1,\ldots,\eps_N)$, $N$ odd, we have
\[ \forall N\ge 5:\;\E[\abs{\med(\eps)}^r]\thicksim N^{-r/2}\text{ and }\E[\abs{\med(\eps)}^{2r}]\thicksim N^{-r}
\]
as well as for $\tau_N\to\infty$ with $\tau_N=o(N^{1/2})$
\[ \limsup_{N\to\infty}\PP\big(2N^{1/2}g_\eps(0)\abs{\med(\eps)}>\tau_N\big)\exp(\tau_N^2/8)\le 2.\]
\end{proposition}

\section{Simulation results}\label{SecSimul}

\begin{figure}[t]
\includegraphics[width=6.5cm]{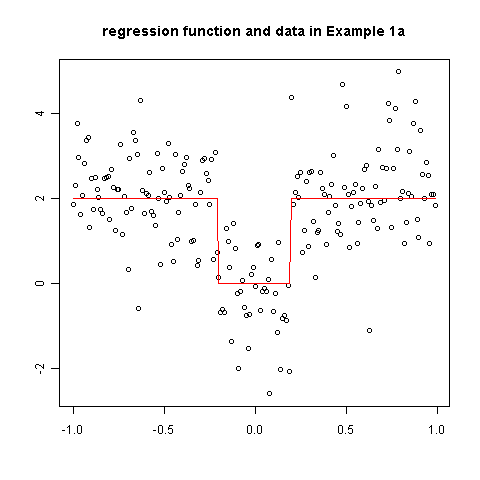}
\includegraphics[width=6.5cm]{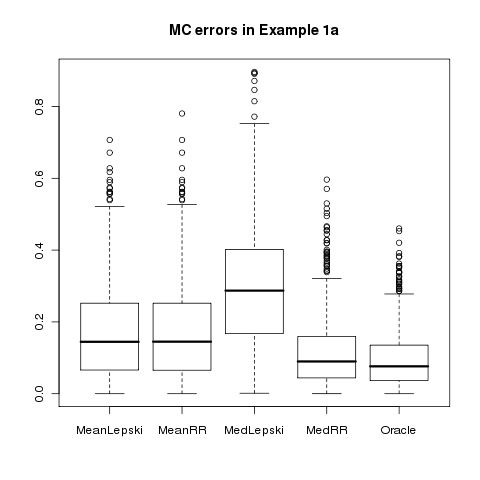}
\caption{Example 1 with Laplace noise: A typical realisation and a box plot of the sample errors in 1000
Monte Carlo runs.}
\end{figure}
\begin{figure}[t]
\end{figure}

We illustrate our procedure by an implementation for median regression on
${\cal X}=[-1,1]$ and the estimation of the regression function at $x=0$. We
simulate $n=200$ equidistant observations $(Y_i)$ with standardized errors
$(\eps_i)$ ($\E[\eps_i]=0$, $\Var(\eps_i)=1$) that are (a) Laplace, (b) normal
and (c) Student t-distributed with three degrees of freedom. The location is
each time estimated by local sample means as well as by local sample medians.
As neighbourhoods we take symmetric intervals $U_k$ around zero containing
$\floor{5^k/4^{k-1}}$ data points. This gives $K=17$ different base estimators.

The calibration of the procedure is performed for Laplace distributed errors
with $r=2$ and $\alpha=1$. The variances $s_j$, $s_{jk}$ of the sample means
are calculated exactly and those of the sample medians are approximated by
their asymptotic values (which are quite close to Monte Carlo values). The
critical values $(z_k)$ are chosen according to the prescription in
\eqref{eqhypzkgen}. This is achieved in both cases, mean and median estimators,
by using the choice in Proposition \ref{propasympzk} with values $\zeta$ that
are calibrated by 10000 Monte Carlo runs for the pure noise situation. It
turned out that this gives almost equally sized error contributions for the
different values $z_k$, as postulated in \eqref{eqhypzk}. The same calibration
principle was applied for the original Lepski procedure with mean and median
estimators.

\begin{figure}[t]
\includegraphics[width=6.5cm]{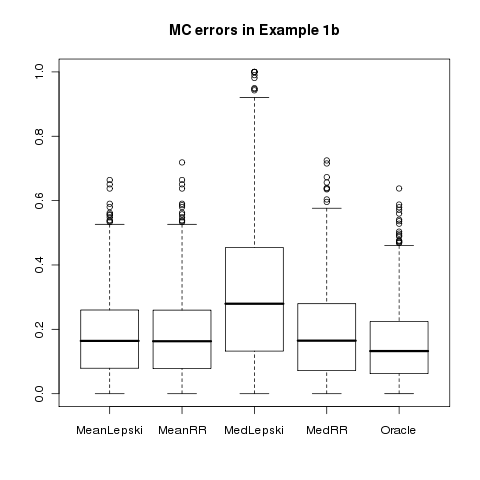}
\includegraphics[width=6.5cm]{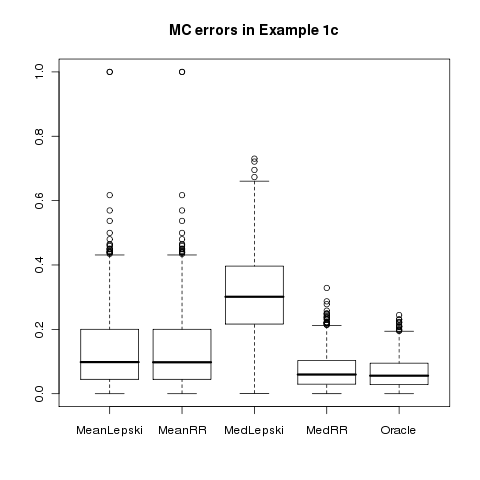}
\caption{Box plot of the sample errors in 1000
Monte Carlo runs for Gaussian (left) and Student t(3) noise (right).}
\end{figure}
\begin{figure}[t]
\end{figure}

As a first example we take a simple change point problem by considering the
regression function $g(x)=0$ for $\abs{x}\le 0.2$ and $g(x)=2$ for
$\abs{x}>0.2$, which can be considered as a toy model for edge detection in
image restauration or for structural breaks in econometrics. In Figure 1 we
show a typical data set in the Laplace case (a) together with box plots for the
absolute error of the different methods in 1000 Monte Carlo repetitions: local
means with Lepski's and with our method, local medians with Lepski's and with
our method and the oracle method, which is just the sample median over
$[-0.2,0.2]=\{x:\,g(x)=0\}$. For exactly the same methods, especially still
calibrated to Laplace errors, Figure 2 presents the results for Gaussian and
heavy-tailed Student t(3) errors.

It is obvious that in all cases Lepski's method applied to sample medians as
base estimators works quite badly. This is due to the fact that this method
stops far too late: the sample median over the complete intervals $U_k$ does
not really 'notice' the jump in the data. In fact, in the Laplace simulation
study the oracle $k=10$ is selected by this method in less than $1\%$ of the
cases while most often ($65\%$) the selection is $k=12$ which yields the $1.5$
times larger window $U_{12}=[-0.29,0.29]$. The methods using the sample mean
estimators perform reasonably well and especially both very similarly. Still,
they are clearly beaten by our median based procedure in cases (a) and (c)
where the median is the more efficient location estimator. It is remarkable
here that we nearly achieve the risk of the oracle median estimator. Even in
the Gaussian case (b) the linear procedures have only minor advantages.
Finally, we notice the robustness property that the calibration with the wrong
error distribution in Figure 2 does not seriously affect the results.

\begin{figure}[t]
\includegraphics[width=6.5cm]{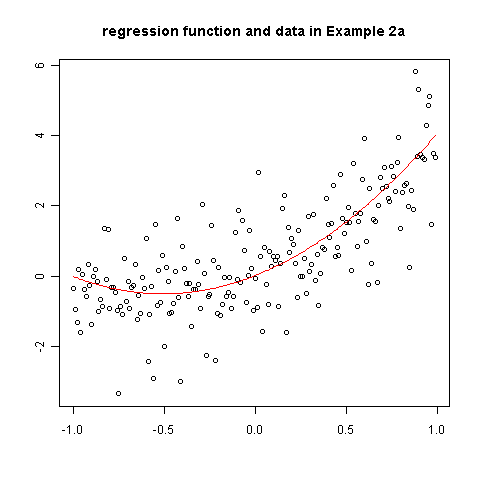}
\includegraphics[width=6.5cm]{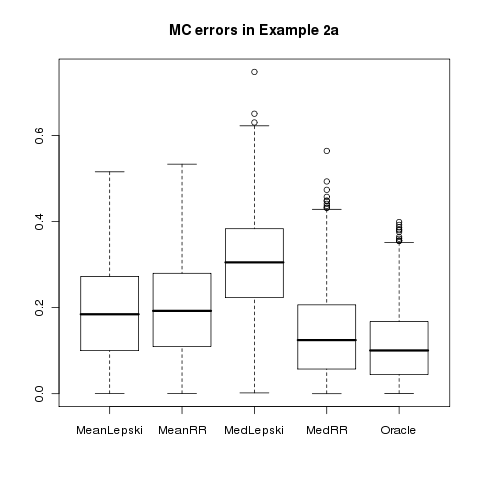}
\caption{Example 2 with Laplace noise. A typical realisation and a box plot of the sample errors in 1000
Monte Carlo runs.}
\end{figure}
\begin{figure}[t]
\end{figure}

In a second example we consider the smooth regression function $g(x)=2x(x+1)$.
Because we are estimating locally around $x=0$, this is a caricature of a
$C^2$-function with $g'(0)=2$ and $g''(0)=4$. Figure 3 shows again a typical
data set and boxplots for the different methods in 1000 Monte Carlo runs under
Laplace errors. This time the oracle choice is the window $[-0.39,0.39]$. Our
median based procedure outperforms the others where the advantage over the
mean-based approaches is again mainly due to the relative efficiency gain of
size $1/\sqrt{2}$ induced by the base estimators in the Laplace model. This
gain, though, is not at all visible when using Lepski's method for selecting
among the sample medians. The results for the error distributions (b) and (c)
resemble those of the first example, we confine ourselves to summarizing the
numerical results for all examples in the following table, each time stating
the Monte Carlo median of the absolute error:

{\centering

\begin{tabular}{|l|r|r|r|r|r|r|}\hline
Ex. & Mean Lepski & Mean RR & Median Lepski & Median RR & Median Oracle\\\hline\hline
1a  & 0.1446 & 0.1450 & 0.2871 & 0.0897 & 0.0763\\\hline
1b  & 0.1640 & 0.1630 & 0.2795 & 0.1647 & 0.1325\\\hline
1c  & 0.0982 & 0.0978 & 0.3012 & 0.0596 & 0.0560\\\hline
2a  & 0.1846 & 0.1924 & 0.3051 & 0.1246 & 0.1005\\\hline
2b  & 0.1808 & 0.1886 & 0.3430 & 0.1586 & 0.1241\\\hline
2c  & 0.2102 & 0.2126 & 0.2455 & 0.1047 & 0.0822\\\hline
\end{tabular}\\
}

\bigskip

Further simulation experiments confirm this picture. Especially for lower values of the moment $r$ our median-based procedure is very efficient, while sometimes for $r=2$ the mean-based procedures profit from less severe outliers in the Monte Carlo runs. In all these experiments the location is equally described by mean and median and we mainly see the efficiency gain of the sample median for non-Gaussian noise. For general quantile regression, however, linear methods do not apply and the standard Lepski procedures based on the nonlinear base estimators will perform badly. Our approach gives significantly better results. The error reductions by a factor of two and more, achieved in the median procedures above, confirm this very clearly.

\section{Application}\label{SecImage}

\begin{figure}[t]
\includegraphics[width=6.5cm]{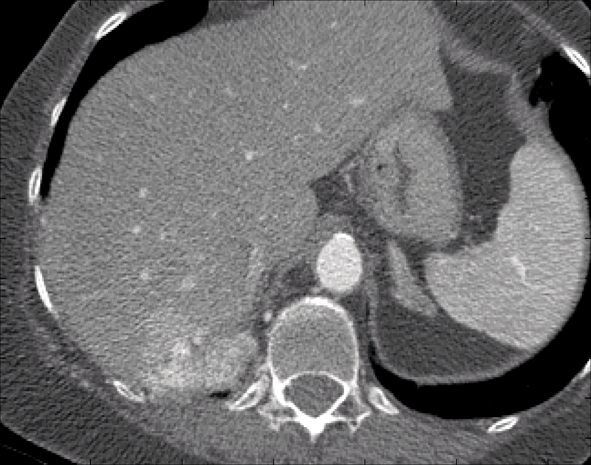}
\includegraphics[width=6.5cm]{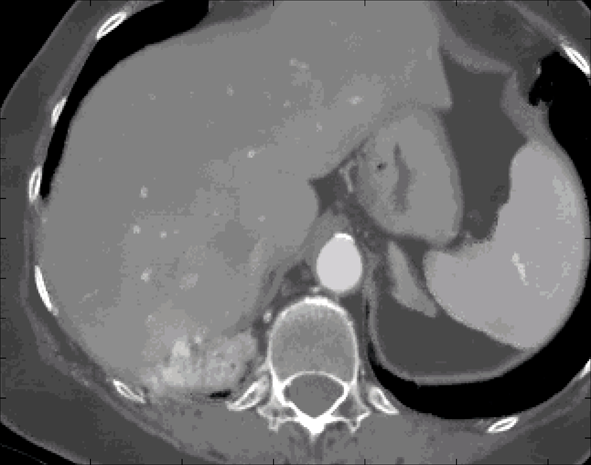}
\caption{CT scan of the upper abdomen, original and result of denoising}
\end{figure}

The proposed procedure is applied to denoise images used in the surveillance of
cancer therapies. In Dynamic Contrast Enhanced Computer Tomography (DCE-CT) a
contrast agent is injected in the human body and its diffusion over time is
observed which is specific for different kinds of cell tissues and allows thus
the surveillance of cancer therapies. For medical reasons the dose of contrast
agent is kept small which leads to a poor signal-to-noise ratio. An analysis of
residuals shows that the observational noise is well modeled by the Laplace
distribution. Moreover, sometimes human movements produce significant outliers.
Therefore local median estimation is employed. Especially for dynamical image
sequences, the denoising is remarkably successful when the same spatial
neighbourhoods are used over the whole observation period. This means that at
each voxel location $x_i$ a vector-valued intensity function $g:{\cal
X}\to\R^K$ is observed under vector-valued noise $\eps_i$. The vector $g(x_i)$
encodes the intensity at time points $(t_1,\dots,t_K)$ recorded at spatial
location $x_i$. Our previously developed procedure perfectly applies to this
situation, we just need a testing procedure between vector-valued local
M-estimators.

Details of the experimental setup and the estimation procedure are discussed in
\cit{RRC} and we merely give a rough description of the setting. A
multiresolution test procedure is applied to compare different vector
estimates. In a first pre-selection step for each voxel $x_i$ we disregard
voxels that are significantly different from $x_i$ and construct then circular
neighbourhoods around $x_i$ consisting only of non-rejected voxels. This allows
geometrically richer neighborhood structures that in practice adapt well to the
structure. Mathematically, the analysis of the algorithm remains the same when
conditioning on the result of this first pre-selection.

\begin{figure}[t]
\includegraphics[width=6.5cm]{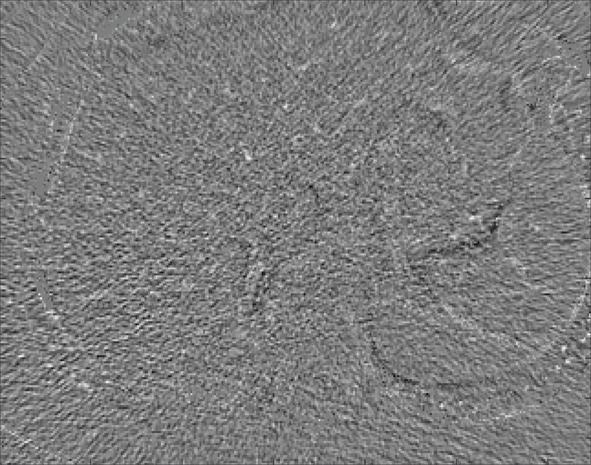}
\includegraphics[width=6.5cm]{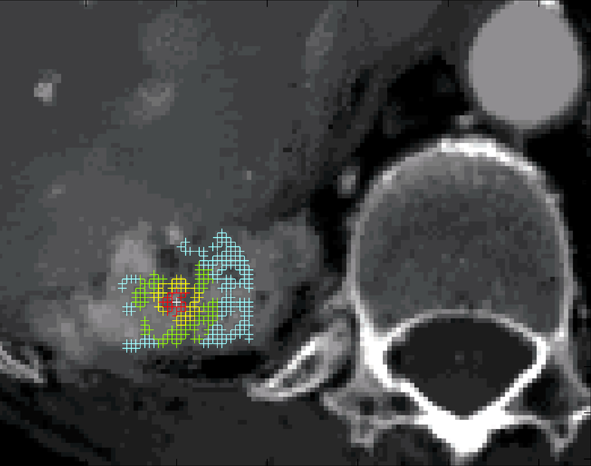}
\caption{CT scan of the upper abdomen, residuals and zoom in denoised image with neighbourhood constructions around
one voxel}
\end{figure}

For the present example we dispose of a DCE-CT sequence of $K=53$ recordings of
$512\times512$-pixel images in the upper abdomen of a cancer patient. In Figure
4 the original image at time step 23 is depicted together with the result of
our denoising procedure. The noise reduction is remarkable while fine
structures like edges are well preserved and not smoothed out. The residuals in
Figure 5(left) show some artefacts due to human body movements and CT radial
artefacts, which our procedure removed as well. In Figure 5(right) a zoom into
Figure 4(right) is shown together with the sequence of neighbourhoods
constructed for one voxel inside the cancerogeneous tissue. The effect of the
pre-selection step is clearly visible by the geometrically adaptive form of the
neighbourhoods. Further results, in particular the denoised dynamics in certain
voxels and an application to automatic clustering of cell tissues are reported
in \cit{RRC}. The generality of our procedure has the potential to provide
statistical solutions in many further applications where spatial inhomogeneity
and robustness are key issues.

\section{Appendix}\label{SecApp}

\subsection{Proof of Proposition \ref{PropTest}}\label{SecApp1}

The asymptotic normality of the sample median $\sqrt{n}\med(Y_1,\ldots,Y_n)\Rightarrow N(0,1/(4f^2(0)))$ is well known
\cite[Corollary 21.5]{vanderVaart} and implies by independence the first asymptotic result.

Since the sample medians in the second case are not independent, we consider their joint distribution using
empirical processes. Let us write $F_\Delta$ for the cumulative distribution function of $f(\cdot-\Delta)$ and denote by $B^1,B^2$ two independent standard Brownian bridges. Then empirical process theory yields by independence
\[\sqrt{n}\Big(\frac1n\sum_{i=1}^n{\bf 1}([Y_i,\infty))-F,\frac1n\sum_{i=n+1}^{2n}{\bf 1}([Y_i,\infty))-F_\Delta\Big)
\Rightarrow (B^1\circ F,B^2\circ F_\Delta).4g
\]
The joint median $\med(Y_i,i=1,\ldots,2n)$ satisfies in terms of the empirical distribution functions
$F^n$ and $F^n_\Delta$ of the two samples
\[ F^n(\med(Y_i,i=1,\ldots,2n))+F_\Delta^n(\med(Y_i,i=1,\ldots,2n))=1.\]
Hence, it can be expressed as the functional $(F^n+F_\Delta^n)^{-1}(1)$ of $(F^n,F_\Delta^n)$, assuming that the
inverse is defined properly (e.g. giving the mean of all admissible values).
Combining two-dimensional versions of Theorem 20.8 and Lemma 21.4 of \cit{vanderVaart}, we infer
\begin{align*}
&\sqrt{n}\Big(\med(Y_i,\,i=1,\ldots,n),\med(Y_i,\,i=1,\ldots,2n)-\Delta/2\Big)\\
&\quad\Rightarrow
 \Big(-(B^1\circ F/f)\circ F^{-1}(1/2), -(B^1\circ F+B^2\circ F_\Delta)/(f+f(\cdot-\Delta))\circ (F+F_\Delta)^{-1}(1)\Big).
\end{align*}
By symmetry of $f$ the right-hand side simplifies to
\[\Big(-B^1(1/2)/f(0), -\big(B^1(F(\Delta/2))+B^2(F_\Delta(\Delta/2))\big)/(2f(\Delta/2))\Big).
\]
Consequently, $\sqrt{n}(2(\med(Y_i,\,i=1,\ldots,2n)-\med(Y_i,\,i=1,\ldots,n))-\Delta)$
is asymptotically normal with mean zero and variance
\begin{align*} \sigma_L^2&=4\E\Big[\Big(-(B^1(F(\Delta/2))+B^2(F(-\Delta/2)))/(2f(\Delta/2))+B^1(1/2)/f(0)\Big)^2\Big]\\
 &= 4\Big(\frac{F(-\Delta/2)(1-F(-\Delta/2))}{4f^2(\Delta/2)}+\frac{F(\Delta/2)(1-F(\Delta/2))}{4f^2(\Delta/2)}
+\frac{1}{4f^2(0)}-\frac{1-F(\Delta/2)}{2f(0)f(\Delta/2)}\Big)\\
&= \frac{2F(\Delta/2)(1-F(\Delta/2))}{f^2(\Delta/2)}+\frac{1}{f^2(0)}-\frac{2(1-F(\Delta/2))}{f(0)f(\Delta/2)}.
\end{align*}
While $\sigma_L^2=\sigma_W^2$ for $\Delta=0$ is straight-forward, we rewrite $\sigma_L^2$ in terms of $R=F/f$ to study the behaviour as $\Delta\to 0$:
\[ \sigma_L^2=2R(\Delta/2)R(-\Delta/2)+4R^2(0)-4R(0)R(-\Delta/2).\]
Because of $R(\Delta/2)-R(-\Delta/2)=\Delta+O(\Delta^2)$ by the Lipschitz property of $f$, we obtain asymptotically
\[ \sigma_L^2=\sigma_W^2+2\Big(\big(R(-\Delta/2)-R(0)\big)^2+
R(-\Delta/2)\big(R(\Delta/2)-R(-\Delta/2)\big)\Big)=\sigma_W^2+\Delta+O(\Delta^2).
\]
This gives $\sigma_L^2=\sigma_W^2(1+2\Delta f(0)+O(\Delta^2f(0)))$.\hfill\qed\\

\subsection{Proof of Proposition \ref{PropMoments}}\label{SecApp2}
 We shall only consider the case of odd $N=2m+1$. Under the conditions of the
proposition \cit{Brownetal} show the following result.

\begin{theorem}\label{ThmCoupling}
For all $m\ge 0$ the sample
$\eps_1,\ldots,\eps_{2m+1}$ can be realised on the same probability
space as a standard normal random variable $Z$ such that
$\med(\eps):=\med(\eps_i,\,i=1,\ldots,2m+1)$ satisfies
\[ \babs{\med(\eps)-\frac{Z}{\sqrt{4(2m+1)}f_\eps(0)}}\le
\frac{C}{2m+1}\Big(1+Z^2\Big) \text{ if }
\abs{Z}\le\delta\sqrt{2m+1},
\]
where $\delta,C>0$ are constants depending on $f_\eps$, but
independent of $m$.
\end{theorem}

The construction and the inequality of the theorem yield with some
constant $C'>0$
\begin{align*}
\E[&\abs{\med(\eps)}^{2r}{\bf 1}(\abs{Z}\le\delta\sqrt{2m+1})]\\
&\le (2^{r-1}\vee
1)\E\Big[\big(4(2m+1)f_\eps(0)^2\big)^{-r}\abs{Z}^{2r}+
\tfrac{C^{2r}}{(2m+1)^{2r}}\Big(1+Z^2\Big)^{2r}\Big]\\
&\le C'(2m+1)^{-r}.
\end{align*}
On the other hand, because of $\eps_i\in L^r$ we have for $z\to\infty$ that
the cdf satisfies $F_\eps(-z)\lesssim \abs{z}^{-r}$ and $1-F_\eps(z)\lesssim
\abs{z}^{-r}$. From the formula for the density of $\med(\eps)$
\[ f_m(z)=
\binom{2m+1}{m+1}(m+1)f_\eps(z)F_\eps(z)^{m}(1-F_\eps(z))^{m}
\]
we therefore infer that
$\norm{\med(\eps)}_{L^{3r}}$ for $m\ge 2$ is finite and uniformly bounded.
Hence, the H\"older inequality gives
\[ \E[\abs{\med(\eps)}^{2r}{\bf 1}(\abs{Z}>\delta\sqrt{2m+1})]\le
\E[\abs{\med(\eps)}^{3r}]^{2/3}\PP(\abs{Z}>\delta\sqrt{2m+1})^{1/3},
\]
which by Gaussian tail estimates is of order
$\exp(-\delta^2(2m+1)/6)$ and thus for $m\to\infty$ asymptotically
negligible. This gives the upper moment bound for $\med(\eps)$, the
lower bound follows symmetrically. The $r$-th moment is bounded by even simpler arguments.

The second assertion follows via quantile coupling from
\begin{align*}
&\PP\big(\sqrt{4(2m+1)}f_\eps(0)\abs{\med(\eps)}>\tau_m\big)\\
&\le
\PP\big(\abs{Z}+\tfrac{C}{\sqrt{2m+1}}(1+Z^2)>\tau_m\big)+\PP\big(\abs{Z}>\delta\sqrt{2m+1}\big)\\
&\le
\PP(2\abs{Z}>\tau_m)+\PP(\abs{Z}>\delta\sqrt{2m+1})\\
&\le 2\exp(-\tau_m^2/8).
\end{align*}
\mbox{}\hfill\text{\qed}\\

%\end{document}
%%%%%%%%%%%%%%%%%%%%%%%%%%%%%%%%%%%%%%%%%%%%%%%%%%%%%%%%%%
\bibliographystyle{dcu}
\thebibliography{99}
\harvarditem{Arcones}{2002}{Arcones} {\sc Arcones, M. A.} (2002). Moderate Deviations for $M$-estimators.
{\sl Test} {\bf 11}(2), 465-500.

\harvarditem{Arias-Castro and Donoho}{2006}{AriasDonoho} {\sc
Arias-Castro, E.} and {\sc D. Donoho} (2006). Does median filtering
truly preserve edges better than linear filtering?, {\sl Preprint}
in Math arXive {\tt math/0612422v1}.

\harvarditem{Brown {\it et al.}}{2008}{Brownetal} {\sc Brown, L. D., T.
Cai} and {\sc H. Zhou} (2008). Robust nonparametric estimation via
wavelet median regression. {\sl Annals of Statistics} {\bf 36}(5), 2055--2084.

\harvarditem{Chu and Hotelling}{1955}{ChuHot} {\sc Chu, J. T.} and
{\sc H. Hotelling} (1955). The moments of the sample median. {\sl
Annals of Math. Statistics} {\bf 26}(4), 593--606.

\harvarditem{Hall and Jones}{1990}{HallJones} {\sc Hall, P.} and
{\sc M. C. Jones} (1990). Adaptive $M$-estimation in nonparametric
regression. {\sl Annals of Statistics} {\bf 18}(4), 1712--1728.

\harvarditem{Huber}{1964}{Huber} {\sc Huber, P. J.} (1964).
Robust estimation of a location parameter. {\sl Annals of Mathematical Statistics}
{\bf 35}(1), 73--101.

\harvarditem{Katkovnik and Spokoiny}{2008}{KatSpok} {\sc Katkovnik,
V.} and {\sc V. Spokoiny} (2008). Spatially adaptive estimation via
fitted local likelihood techniques. {\sl IEEE Transactions on Signal Processing} {\bf 56}(3), 873--886.

\harvarditem{Koenker}{2005}{Koenker} {\sc Koenker, R.}  (2005). {\sl Quantile Regression}.
Econometric Society Monographs 38, Cambridge University Press.

%\harvarditem{Kovac}{2002}{Kovac} {\sc Kovac, A.}  (2002). Robust
%nonparametric regression and modality. in {\sl Developments in
%Robust Statistics}, R. Dutter, P. Filzmoser, U. Gather, P. Rousseeuw
%ed., Physica, Heidelberg, 218 -- 227.

\harvarditem{Lepski}{1990}{Lepski} {\sc Lepskii, O. V.} (1990).
A problem of adaptive estimation in Gaussian white noise. {\sl Theory Probab. Appl.} {\bf 35}(3), 454--466. Translated from {\sl Teor. Veroyatnost. i Primenen.}
{\bf 35}(3) (1990), 459--470.

\harvarditem{Lepski {\it et al.}}{1997}{Lepskietal} {\sc Lepski, O., E.
Mammen} and {\sc V. Spokoiny} (1997). Optimal spatial adaptation to
inhomogeneous smoothness: an approach based on kernel estimates with
variable bandwidth selectors. {\sl Annals of Statistics} {\bf
25}(3), 929--947.

\harvarditem{Massart}{2007}{Massart} {\sc Massart, P.} (2005). Concentration inequalities and model selection.  Ecole d'Et\'e de Probabilit\'es de Saint-Flour XXXIII -- 2003, {\sl Lecture Notes in Mathematics 1896}, Springer, Berlin.

%\harvarditem{Polzehl and Spokoiny}{2000}{PolSpok00} {\sc Polzehl, J.} and {\sc V. Spokoiny}
%(2000). Adaptive weights smoothing with applications to image denoising. {\sl J. Royal Statist. Society Ser. B} {\bf %62}, 335--354.

\harvarditem{Polzehl and Spokoiny}{2003}{PolSpok03} {\sc Polzehl, J.} and {\sc V. Spokoiny}
(2003). Image denoising: pointwise adaptive approach. {\sl Annals of Statistics} {\bf 31}, 30--57.

\harvarditem{Portnoy}{1997}{Portnoy} {\sc Portnoy, S.} (1997). Local asymptotics for quantile smoothing splines. {\sl Annals of Statistics} {\bf 25}(1), 414--434.

\harvarditem{Rozenholc {\it et al.}}{2009}{RRC} {\sc Rozenholc, Y., M. Rei\ss, D. Balvay} and {\sc C.-A. Cuenod}(2009). Growing time-homogeneous neighbourhoods for denoising and clustering dynamic contrast enhanced-CT sequences, {\sl Preprint Universit\'e Paris V}.

\harvarditem{Spokoiny and Vial}{2009}{SpokVial} {\sc Spokoiny, V.}
and {\sc C. Vial} (2009). Parameter tuning in pointwise adaptation using a propagation approach, {\sl Annals of Statistics}, to appear.

\harvarditem{Truong}{1989}{Truong} {\sc Truong, Y.K.} (1989).
Asymptotic Properties of Kernel Estimators Based on Local Medians,
{\sl Annals of Statistics} {\bf 17}(2), 606--617.

\harvarditem{Tsybakov}{2004}{Tsybakov}  {\sc Tsybakov, A. B.} (2004). Optimal aggregation of classifiers in statistical learning. {\sl Annals of Statistics} {\bf 32}(1) 135--166.

\harvarditem{van de Geer}{2003}{vandeGeer} {\sc van de Geer, S.} (2003). Adaptive quantile regression, in  {\sl Recent Trends in Nonparametric Statistics} (Eds. M.G. Akritas and D.N. Politis), Elsevier Science, 235--250.

\harvarditem{van der Vaart}{1998}{vanderVaart} {\sc van der Vaart, A.} (1998).
{\sl Asymptotic Statistics}, Cambridge University Press.

\end{document}